\documentclass[a4paper,12pt,reqno]{amsart}
\usepackage[body={17cm,21cm}]{geometry}
\usepackage[initials]{amsrefs}

\usepackage{amssymb,amscd}
\usepackage[all]{xy}

\usepackage[mathscr]{eucal}
\usepackage{enumerate}

\numberwithin{equation}{section}
\theoremstyle{plain}
\newtheorem{thm}[equation]{Theorem}
\newtheorem{cor}[equation]{Corollary}
\newtheorem{lem}[equation]{Lemma}
\newtheorem{prop}[equation]{Proposition}

\newtheorem{definition}[equation]{Definition}

\newtheorem{rem}[equation]{Remark}

\theoremstyle{definition}

\theoremstyle{remark}

\newtheorem{claim}[equation]{Claim}

\DeclareMathOperator{\Gal}{Gal}

\def\inv{{\rm inv\,}}

\def\A{{\mathbb A}}
\def\Gal{{\rm Gal}}

\DeclareFontFamily{U}{wncy}{}
\DeclareFontShape{U}{wncy}{m}{n}{%
<5>wncyr5%
<6>wncyr6%
<7>wncyr7%
<8>wncyr8%
<9>wncyr9%
<10>wncyr10%
<11>wncyr10%
<12>wncyr6%
<14>wncyr7%
<17>wncyr8%
<20>wncyr10%
<25>wncyr10}{}
\DeclareMathAlphabet{\cyr}{U}{wncy}{m}{n}
\begin{document}

\title[Strong approximation for certain quadric fibrations with compact fibers]
{Strong approximation for certain quadric fibrations with compact fibers}

\author{Fei XU}

\address{School of Mathematical Sciences \\
Capital Normal University \\
Beijing 100048, China}

\email{xufei@math.ac.cn}

\address{}

\email{}

\date{\today}

\maketitle

\section{Introduction}

In \cite{Wat}, G.L.Watson investigated the local-global principle over $\Bbb Z$ for the following equation
\begin{equation} \label{watson}
q(x_1,\cdots,x_n)+\sum_{i=1}^n a_i(t)x_i+b(t)=0
\end{equation}
where $q(x_1,\cdots,x_n)$ is a quadratic form over $\Bbb Z$ and $a_1(t), \cdots, a_n(t)$ and $b(t)$ are polynomials over $\Bbb Z$.
The variety defined by (\ref{watson}) is isomorphic to the variety defined by
\begin{equation} \label{ctx} q(x_1, \cdots, x_n)=p(t) \end{equation}
over $\Bbb Q$, where $p(t)$ is a polynomial over $\Bbb Q$. The study of strong approximation of equation (\ref{ctx}) will provide the solvability of equation (\ref{watson}) by choosing the equation (\ref{watson}) as an integral model. In \cite{CTX1}, J.-L. Colliot-Th\'el\`ene and the author proved that the equation (\ref{ctx}) satisfies the strong approximation with Brauer-Manin obstruction when the quadratic form $q$ is indefinite with $n\geq 3$. In \cite{CTH}, J.-L. Colliot-Th\'el\`ene and D.Harari extended this result to very general fibration of homogeneous spaces of linear algebraic groups with isotropic assumption for the fibers over $\Bbb R$-points. By such fibration method, one needs enough fibers satisfying the strong approximation with Brauer-Manin obstruction. One of the necessary conditions of strong approximation with Brauer-Manin obstruction for subvarieties of affine varieties is that the set of $\Bbb R$-points is not compact. In this paper, we will show that strong approximation with Brauer-Manin obstruction holds for (\ref{ctx}) with $n\geq 3$ if and only if $\Bbb R$-points of total spaces is not compact. When the quadratic form $q$ is definite, this provides the example that none of fibers satisfies strong approximation with Brauer-Manin obstruction. As application, the main result in \cite{Wat} is the immediate consequence of our result. Moreover, in order to prove our result, we develop the representation theory of quadratic Diophantine equations and explain that the representability of quadratic polynomials is equivalent to the classical result of representability of quadratic forms with congruent conditions and extend Theorem 2.1 in \cite{HJ} to general number fields.

Notation and terminology are standard. Let $F$ be a number field, $\frak o_F$ be the ring of integers of
$F$ and $\Omega_F$ be the set of all primes in $F$. For each $v\in \Omega_F$, let $F_v$ be the completion of $F$ at $v$.
Let  $\infty_F$ be  the set of  archimedean primes in $F$ and write
$v<\infty_F$ for $v\in \Omega_F\setminus \infty_{F}$.
For each $v<\infty_F$, let $\frak o_v$ be the completion
of $\frak o_F$ at $v$ and $\pi_v$ is the uniformizer of
$\frak o_{v}$ with $ord_v(\pi_v)=1$ and assume $ord_v(0)=+ \infty$. Write $\frak o_{v}=F_v$ for
$v\in \infty_F$.
For any finite subset $T$ of $\Omega_F$, let $F_{T}=\prod_{{v} \in T}F_{v}$.
For any finite subset $S$ of $\Omega_F$ containing $\infty_F$,
 the $S$-integers are defined to be elements in $F$ which are integral outside $S$ and denoted by $\frak o_S$. Let $\Bbb A_F$ be the adelic group of $F$
with its usual topology.
For any finite set $T$ of $\Omega_F$, one   defines $\Bbb A_{F}^T \subset  (\prod_{v\not \in T} F_v)$
equipped with analogous   adelic
  topology. The  natural projection which omits the  $T$-coordinates  defines  homomorphism of rings
$  \Bbb A_F \to \Bbb A_{F}^T$. For any scheme of finite type $X$ over $F$ this induces a map
$$ pr^T  :  X(\Bbb A_{F})   \to   X(\Bbb A_{F}^T).$$
Let $Br(X)=H_{et}^2(X, \Bbb G_m)$ and
$$X_F(\mathbb A_F)^{Br(X)}= \{ (x_v)_v\in X_F(\mathbb A_F):
  \sum_{v\in \Omega_F }\inv_v(\xi(x_v))=0 \ \text{ for all $\xi\in Br(X)$}
  \}.$$ Then $$ X(F) \subseteq X_F(\mathbb A_F)^{Br(X)} \subseteq X_F(\mathbb A_F) $$ by class field theory.

\begin{definition} Let $X$ be a scheme of finite type over $F$. One says that strong approximation with Brauer-Manin obstruction off $T$ holds for $X$ if $X(F)$ is dense in $pr^T(X(\Bbb A_F)^{Br(X)})$ under the diagonal map.
 \end{definition}

The paper is organized as follows. In \S 2, we set up spinor genus theory for quadratic Diophantine equations briefly. In \S 3,  we proved strong approximation with Brauer-Manin obstruction off $\infty_F$ for the variety defined by (\ref{ctx}). As application, we explain the equivalence between representability of quadratic Diophantine polynomials and representability of quadratic forms with congruent condition for definite case and extend Theorem 2.1 in \cite{HJ} over a number field.
\bigskip

\section{Quadratic Diophantine equations}\label{qde}

Arithmetic theory of quadratic forms is a classical topic and has been extensively studied by various methods, for example in \cite{Ei}, \cite{OM}, \cite{Cas} and \cite{Ki}. In order to establish the strong approximation for (\ref{ctx}) with definite quadratic form $q$, we need to extend the classical results for quadratic forms to general quadratic polynomials. Such generalization has already been considered by Kneser in \cite{Kn} and Watson in \cite{Wat60}. Recently,  for various purpose, such generalization has been studied by Shimura in \cite{Shi1}, \cite{Shi2} and \cite{Shi3}, Sun in \cite{Sun07}, Colliot-Th\'el\`ene and Xu in \cite{CTX}, Wei and Xu in \cite{WX} and Chan and Oh in \cite{CO} and so on. In this section, we briefly set up some basic properties for quadratic Diophantine equations which we need in next sections.

Let $V$ be a non-degenerated quadratic space over $F$
with the associated symmetric bilinear map $$ \langle, \rangle : \ \ V\times V\longrightarrow
F \ \ \ \text{with} \ \ \ q(x)=\langle x,x \rangle$$ for any $x\in V$ and the special orthogonal group $$SO(V)=\{ \sigma \in GL(V):  \ \ q(\sigma
x)=q(x) \ \ \text{for all} \ \  x\in V  \ \ \text{and} \ \ det(\sigma)=1 \}  $$ with the double cover $Spin(V)$ satisfying
\begin{equation} \label{galois-coh} 1\rightarrow \mu_2 \rightarrow Spin(V) \xrightarrow{\iota} SO(V) \xrightarrow{\theta} H^1(F,\mu_2)\cong F^\times/(F^\times)^2  \end{equation}  by Galois cohomology. Let
$SO_\A(V)$, $Spin_\A (V)$ and  $GL_\A(V)$ be the adelic groups of $SO(V)$, $Spin(V)$ and
$GL(V)$ respectively.

A finitely generated $\frak o_F$-module $L$ in $V$ is called an $\frak o_F$-lattice if $FL=V$. We denote the local
completion of a lattice $L$ inside the local completion $V_v$ of $V$ at
$v$ by $L_v$ for $v< \infty_F$ and $L_v=V_v$ for
$v\in \infty_F$. The classical theory of integral quadratic forms can be formulated in terms of lattices such as \cite{OM} and \cite{Ki}. Most of theory can be extended to general quadratic equations in terms of lattice translations, i.e. $L+u_0$ for some $u_0\in V$.

The local representation of integral quadratic forms has been extensively studied in  \cite{OM58} and \cite{Rie}. For general quadratic equations, one has the following result.

\begin{lem} \label{opencom} Let $L$ be a lattice in $V$ and $u_0\in V$. If $u_0\not\in L_v$ for $v<\infty_F$, then $q(L_v+u_0)$ is an open compact subset of $F_v$. \end{lem}
\begin{proof} Since the quadratic form $q$ on $V_v$ defines a morphism
$$ q: {\bf A}^n \rightarrow {\bf A}^1 $$ of affine spaces over $F_v$, one obtains that $q$ is smooth over the open sub-scheme ${\bf A}^n \setminus \{(0, \cdots, 0)\}$. By the inverse function theorem (see Thm.3.2 and Prop.3.3 in \cite{PR}), one has the continuous map
 $$ q: \ \ V_v \setminus \{0\} \rightarrow F_v $$ is open. If $u_0\not\in L_v$, then $L_v+u_0$ is an open and compact subset of $V_v \setminus \{0\}$. Therefore $q(L_v+u_0)$ is an open compact subset of $F_v$.
\end{proof}

\begin{cor} \label{neig} Let $L$ be a lattice in $V$ and $u_0\in V$. If $u_0\not\in L_v$ for $v<\infty_F$, there is an open subgroup $U$ of $\frak o_v^\times$ such that $\alpha U \subseteq q(L_v+u_0)$ for all $\alpha\in q(L_v+u_0)$.
\end{cor}
\begin{proof} For any $x\in q(L_v+u_0)$, there is a positive integer $a_x$ such that $$ x+ \pi_v^{a_x} \frak o_v  \subseteq q(L_v+u_0) $$ by Lemma \ref{opencom}. By compactness of $q(L_v +u_0)$, one has
$$ q(L_v+u_0) = \bigcup_{i=1}^n (x_i+ \pi_v^{a_{x_i}} \frak o_v) .$$ Let $$ a = \max_{1\leq i\leq n}  \{ a_{x_i}, \ \ a_{x_i}- ord_v(x_i)\} \ \ \ \text{and} \ \ \  U =1+\pi_v^a \frak o_v$$ be an open subgroup of $\frak o_v^\times$. For any $\alpha\in q(L_v+u_0)$, there is $1\leq i_0 \leq n$ such that $$\alpha\in x_{i_0} + \pi_v^{a_{x_{i_0}}} \frak o_v . $$ Therefore
$$ \alpha U \subseteq (x_{i_0}+ \pi_v^{a_{x_{i_0}}} \frak o_v)(1+ \pi_v^a \frak o_v) \subseteq x_{i_0} + \pi_v^{a_{x_{i_0}}} \frak o_v  \subseteq q(L_v+u_0) $$ as required.  \end{proof}

By Lemma 3.2 in \cite{WX} or Lemma 4.2 in \cite{CO}, one can define the action of $GL_\A(V)$ on
$L+u_0$. 

\begin{definition} Let $L$ be an $\frak o_F$-lattice and $u_0\in V$. We define

$$gen(L+ u_0)= \text{the orbit of $(L+u_0)$ under the action of $SO_\A(V )$} $$ is called the genus of $(L+u_0)$.

$$spn(L+u_0)= \text{the orbit of $(L+u_0)$ under the action of $SO(V) \iota (Spin_\A (V))$} $$ is called the spinor genus of $(L+u_0)$.

$$cls(L+ u_0) = \text{the orbit of $(L+ u_0)$ under the action of $SO(V)$ }$$ is called the class of $(L+u_0)$.
\end{definition}

It is clear that
$$ cls(L+u_0)\subseteq spn(L+u_0) \subseteq gen(L+u_0) . $$

Set
$$SO(L+ u_0)=\{ \sigma \in SO(V): \ \sigma (L+ u_0)=(L+ u_0) \}$$ and
$$SO_\A(L+ u_0)=\{ \sigma \in SO_\A(V): \ \sigma
(L+ u_0)=(L+ u_0) \}$$ and
$$SO(L_v+ u_0)=\{ \sigma \in SO(V_v): \ \sigma
(L_v+ u_0)=(L_v+ u_0) \}$$  for any $v< \infty_F$.

\begin{prop} \label{numspn} If $rank(L)\geq 3$, then the number of spinor genera in $gen(L+u_0)$ is given by $$[\Bbb I_F: F^\times \prod_{v\in \Omega_F} \theta(SO(L_v+u_0))] $$
where $\theta$ is the map induced by (\ref{galois-coh}).
\end{prop}
\begin{proof} The number of spinor genera is given by
$$[SO_\Bbb A(V): SO(V) \iota(Spin_\Bbb A(V)) SO_\Bbb A(L+u_0)]. $$ Applying the map $\theta$ induced by (\ref{galois-coh}), one has the bijection
$$ SO_\Bbb A(V)/SO(V) \iota(Spin_\Bbb A(V)) SO_\Bbb A(L+u_0) \cong \Bbb I_F/F^\times \prod_{v\in \Omega_F} \theta(SO(L_v+u_0))$$ by 102:7 in \cite{OM}. \end{proof}

\begin{definition} Let $L_v$ be a lattice in $V_v$ and $x\in V_v$ for $v<\infty_F$. The coefficient $\frak C_x$ of $x$ in $L_v$ is defined as the following fractional ideal $$\frak C_x= \{ \alpha \in F_v: \ \alpha x\in L_v \}. $$

We call $x$ is primitive in $L_v$ if $\frak C_x=\frak o_v$ and denote $x\in^* L_v$.

If there is $x\in^* L_v$ such that $q(x)=\alpha$, we call $\alpha$ is represented primitively by $L_v$ and denoted by $\alpha \xrightarrow{*} L_v$.
\end{definition}

\begin{definition} \label{rep} Suppose $L$ is an $\frak o_F$-lattice in a quadratic space $V$ and $u_0\in V$ and $\alpha\in F$.

We call $\alpha$ is represented by $gen(L+u_0)$ (resp. $spn(L+u_0)$ and  $cls(L+u_0)$) denoted by
$$ \alpha\rightarrow gen(L+u_0) \ \ \ (\text{resp.} \ \  spn(L+u_0) \ \ \text{and} \ \ cls(L+u_0))  $$
if there is $(L'+u_0')\in gen(L+u_0)$ (resp. $spn(L+u_0)$ and $cls(L+u_0)$) such that $\alpha \in q(L'+u_0')$.

Let $T$ be a finite set of primes of $F$ such that $u_0\in L_v$ for $v\not\in T$. We call $\alpha$ is represented primitively by $gen(L+u_0)$ (resp. $spn(L+u_0)$ and  $cls(L+u_0)$) outside $T$ denoted by
$$ \alpha\xrightarrow{*}_T  gen(L+u_0) \ \ \ (\text{resp.} \ \  spn(L+u_0) \ \ \text{and} \ \ cls(L+u_0))  $$
if there is $(L'+u_0')\in gen(L+u_0)$ (resp. $spn(L+u_0)$ and $cls(L+u_0)$) and $x\in L'+u_0'$ such that $q(x)=\alpha$ and $x\in^* L'_v$ for $v\not\in T$.
\end{definition}

The following lemma extends some classical results for usual lattices to lattice translations.

\begin{lem} \label{x} Let $L$ be a lattice in $V$ with $rank(L)\geq 3$ and $u_0\in V$. Suppose $x\in L+u_0$ with $q(x)=\alpha$. One has the following bijections
$$ \aligned &  \{ spn(L'+u_0')\in gen(L+u_0): \alpha \rightarrow spn(L'+u_0') \}  \\
 \xrightarrow{\cong}    \ \ & X_\Bbb A(x, L+u_0) SO(V)\iota(Spin_\Bbb A(V)) /SO(V)\iota(Spin_\Bbb A(V))SO_\Bbb A(L+u_0) \\
 \xrightarrow{\cong}  \ \ & F^\times \theta(X_\Bbb A(x, L+u_0))/F^\times \prod_{v\in \Omega_F} \theta(SO(L_v+u_0))
 \endaligned $$
where $$ X_\Bbb A(x,L+u_0)= \{ \sigma_\Bbb A \in SO_\Bbb A(V): \ \ x\in \sigma_\Bbb A \circ (L+u_0) \} .$$ Moreover, if $T$ is a finite set of primes of $F$ such that $u_0\in L_v$ for $v\not\in T$, then one has the following bijections
$$ \aligned &  \{ spn(L'+u_0')\in gen(L+u_0): \alpha \xrightarrow{*}_T spn(L'+u_0') \}  \\
 \xrightarrow{\cong}    \ \ & X_T^*(x, L+u_0) SO(V)\iota(Spin_\Bbb A(V)) /SO(V)\iota(Spin_\Bbb A(V))SO_\Bbb A(L+u_0) \\
 \xrightarrow{\cong}  \ \ & F^\times \theta(X_T^*(x, L+u_0))/F^\times \prod_{v\in \Omega_F} \theta(SO(L_v+u_0))
 \endaligned $$
where $$ X_T^*(x,L+u_0)= \{ (\sigma_v) \in SO_\Bbb A(V): \ \ x\in (\sigma_v) \circ (L+u_0) \ \ \text{and} \ \ x\in^* \sigma_v L_v \ \ \text{for} \ v\not\in T \} .$$
\end{lem}

\begin{proof} If $\alpha \rightarrow spn(L'+u_0')$, there are $\sigma\in SO(V)$ and $\tau_\Bbb A \in \iota(Spin_\Bbb A(V))$ such that $$x\in \sigma \tau_\Bbb A \circ (L'+u_0')$$ by Witt Theorem (see \cite{OM} \S 42 F). Since $spn(L'+u_0')\in gen(L+u_0)$, one has $\varrho_\Bbb A\in SO_\Bbb A(V)$ such that $\varrho_\Bbb A \circ (L+u_0)=L'+u_0'$. Therefore one has $$\sigma\tau_\Bbb A \varrho_\Bbb A \in X_\Bbb A(x,L+u_0) $$ and defines the first map by sending $$spn(L'+u_0') \mapsto \sigma\tau_\Bbb A \varrho_\Bbb A SO(V)\iota(Spin_\Bbb A(V))SO_\Bbb A(L+u_0) . $$
It is straight forward to verify that the above map is well-defined and the inverse map is given by sending
$$ \sigma_\Bbb A SO(V)\iota(Spin_\Bbb A(V))SO_\Bbb A(L+u_0) \mapsto spn(\sigma_\Bbb A \circ (L+u_0)) . $$
Applying $\theta$ in (\ref{galois-coh}), one obtains the second bijection by 102:7 in \cite{OM}.

If $\alpha \xrightarrow{*}_T spn(L'+u_0')$, then $x\in^* \sigma \tau_v (L'_v)$ for $v\not\in T$ in the above argument. Therefore $$\sigma\tau_\Bbb A \varrho_\Bbb A \in X_T^*(x,L+u_0) $$ and the rest of the proof follows from the exact same argument as above.
\end{proof}

\begin{cor} \label{spn} Let $L$ be a lattice in $V$ and $u_0\in V$ and $\alpha\in F^\times$. Suppose one of the following conditions holds

 i) $rank(L)\geq 4$; or  $rank(L)=3$ and $-\alpha \cdot det(L)\in (F^\times)^2$.

ii) $rank(L)= 3 $ and there is a non-dyadic finite prime $v_0$ such that $L_{v_0}$ is unimodular and $u_0\in L_{v_0}$ with
$ord_{v_0}(\alpha)>0$ and $-\alpha \cdot det(L) \not \in (F_{v_0}^\times)^2$.

If $\alpha \rightarrow gen(L+u_0)$, then every spinor genus in $gen(L+u_0)$ represents $\alpha$.
Moreover, if $T$ is the finite set of primes of $T$ such that $u_0\in L_v$ for $v\not\in T$ and $v_0\in T$ in case ii) and $\alpha \xrightarrow{*}_T gen(L+u_0)$, then every spinor genus in $gen(L+u_0)$ represents $\alpha$ primitively outside $T$.
\end{cor}

\begin{proof} Since $\alpha \rightarrow gen(L+u_0)$, there are $(K+w_0)\in gen(L+u_0)$ and $x\in (K+w_0)$ such that $q(x)=\alpha$.
Let $V=Fx\perp W$. Then $$ SO_\Bbb A(W) \subseteq X_\Bbb A(x, K+w_0)  \ \ \ \text{and} \ \ \ [\Bbb I_F : F^\times \theta (SO_\Bbb A(W))] \leq 2 $$ by Lemma \ref{x} and 91:6 in \cite{OM}. Moreover
$$ [\Bbb I_F : F^\times \theta (SO_\Bbb A(W))] = 2 \ \ \ \text{if and only if} \ \ \ -det(W)\not\in (F^\times)^2 \ \ \ \text{and} \ \ \ dim(W)=2 $$
with
$$ \Bbb I_F / F^\times \theta (SO_\Bbb A(W)) \cong Gal(F(\sqrt{-det(W)})/F) $$ by the Artin map.

If condition $i)$ holds, one already has
$$  \Bbb I_F= F^\times \theta (SO_\Bbb A(W))=F^\times \theta(X_\Bbb A(x, K+w_0 ))$$ and every spinor genera in $gen(L+u_0)$ represents $\alpha$ by Lemma \ref{x} and Prop.\ref{numspn}.

If condition $ii)$ holds, one has
$$ F^\times \theta(SO_\Bbb A(W)) \neq F^\times \theta( X_\Bbb A(x, K+w_0)) \ \ \ \text{if and only if} \ \ \ \theta(SO_\Bbb A(W)) \neq \theta (X_\Bbb A(x, K+w_0)) $$ by (5.6) Proposition of Chapter VI in \cite{N}. By Theorem 5.1 in \cite{HSX} at $v_0$, one has
$$ \theta(X_{v_0}(x, K+w_0))\neq  \theta(SO(W_{v_0})) $$ where $X_{v_0}(x, K+w_0)$ is the $v_0$-component of $X_\Bbb A(x, K+w_0)$. Therefore one still has $$  \Bbb I_F=F^\times \theta(X_\Bbb A(x, K+w_0))$$ and every spinor genera in $gen(L+u_0)$ represents $\alpha$ by Lemma \ref{x} and Prop.\ref{numspn}.

If $\alpha \xrightarrow{*}_T gen(L+u_0)$, one can further assume that $x\in^* K_v$ for $v\not\in T$. Therefore $$SO_\Bbb A(W) \subseteq X_T^*(x, K+w_0)$$ and the result follows from the same argument as above by replacing $X_\Bbb A(x, K+w_0)$ with $X_T^*(x, K+w_0)$. \end{proof}

The following lemma extends some ideas for usual lattices in \cite{Xu} to lattice translations.

\begin{lem}\label{modify} Let $L$ be a lattice with $rank(L)=3$ in $V$ and $u_0\in V$. Suppose $v_0$ is a finite non-dyadic prime such that $L_{v_0}$ is unimodular with $u_0\in L_{v_0}$ and $t\in \frak o_F$ such that $$ord_v(t)=0 \ \ \ \text{and} \ \ \ (t-1)u_0 \in L_v $$ for all finite primes $v$ except $v=v_0$ and $-\alpha \cdot det(L)\in (F_{v_0}^\times)^2$. If $\alpha \rightarrow spn(L+u_0)$ and $t^{-2k}\alpha \rightarrow gen(L+u_0)$ for a positive integer $k$, then $t^{-2k}\alpha \rightarrow spn(L+u_0)$.

Moreover, if $T$ is a finite set of primes of $F$ containing $v_0$ such that $u_0\in L_v$ for $v\not\in T$ and $\alpha \xrightarrow{*}_T spn(L+u_0)$ and $t^{-2k}\alpha \rightarrow gen(L+u_0)$ for a positive integer $k$, then $t^{-2k}\alpha \xrightarrow{*}_T spn(L+u_0)$.
\end{lem}

\begin{proof} By Corollary \ref{spn}, one only needs to consider $-\alpha \cdot det(L) \not \in (F^\times)^2$. Without loss of generality, we can assume that there is $x\in L+u_0$ such that $q(x)=\alpha$. Let $y=t^{-k} x$. Since $t^{-2k}\alpha \rightarrow gen(L+u_0)$, one has $\sigma_{v_0}\in SO(V_{v_0})$ such that $y\in \sigma_{v_0} (L_{v_0}+u_0)$. Define
$$\sigma_v = \begin{cases} \sigma_{v_0} \ \ \ & \text{if $v=v_0$} \\
1 \ \ \ & \text{otherwise} \end{cases} $$ and $K+w_0=(\sigma_v)_{v\in \Omega_F}\circ (L+u_0)$. Then $y\in K+w_0$ and the spinor genera in $gen(L+u_0)$ which represent $t^{-2k}\alpha$ are given by
$$ F^\times \theta(X_\Bbb A(y, K+w_0))/F^\times \prod_{v\in \Omega_F} \theta(SO(K_v+w_0)) $$ by Lemma \ref{x}. Let $E=F(\sqrt{-\alpha \cdot det(L) })$. Since $v_0$ splits completely in $E/F$, one has $$ \theta ((\sigma_v)_{v\in \Omega_F}^{-1}) \in N_{E/F}(\Bbb I_E) \subseteq \theta(X_\Bbb A(y, K+w_0)). $$ This implies that
$$(\sigma_v)_{v\in \Omega_F}^{-1}SO(V)\iota(Spin_\Bbb A(V))SO_\Bbb A(K+w_0) \in X_\Bbb A(y, K+w_0) SO(V)\iota(Spin_\Bbb A(V)) $$ and
$$ spn((\sigma_v)_{v\in \Omega_F}^{-1}\circ (K+w_0))= spn(L+u_0) $$ represents $t^{-2k}\alpha$ by Lemma \ref{x}.

If $\alpha \xrightarrow{*}_T spn(L+u_0)$, we can further assume that $x\in^* L_v$ for $v\not\in T$ which implies that $y\in^* K_v$ for $v\not\in T$ in the above argument. The result follows from the same argument as above by replacing $X_\Bbb A(x, K+w_0)$ with $X_T^*(x, K+w_0)$. \end{proof}

The following proposition is based on Kneser's arithmetic isotropic idea. Several variants have already appeared, for example Lemma 1.2 in \cite{HKK}.

\begin{prop}\label{arithiso}  Let $L$ be a lattice in $V$ with $dim(V)\geq 3$ and $u_0\in V$. Suppose $v_0$ is a non-dyadic finite prime such that $L_{v_0}$ is unimodular and $u_0\in L_{v_0}$.  Then there is a positive number $h=h(L,u_0, v_0)$ depending only on $L, u_0$ and $v_0$ such that

i) if $ \alpha \rightarrow spn(L+u_0)$ and $ord_{v_0}(\alpha)\geq h$ with $\alpha\in F^\times$, one has $\alpha \rightarrow cls(L+u_0)$.

ii) if $T$ is a finite set of primes of $F$ containing $v_0$ and satisfying $u_0\in L_v$ for $v\not\in T$ and $ \alpha \xrightarrow{*}_T spn(L+u_0)$ and $ord_{v_0}(\alpha)\geq h$ with $\alpha\in F^\times$, one has $\alpha \xrightarrow{*}_T cls(L+u_0)$.
\end{prop}
\begin{proof} Since the class number of $\frak o_F$ is finite, the suitable power of the prime $v_0$ is principal ideal. Then there is $t_0\in \frak o_F$ such that $ord_{v_0}(t_0)>0$ and $ord_v(t_0)=0$ for all finite $v\neq v_0$.

  Let $T_1=\{ v<\infty_F: \ u_0\not\in L_v \}$. For any $v\in T_1$, there is an open subgroup $U_v$ of $o_v^\times$ such that $\alpha U_v \subseteq q(L_v+u_0)$ for all $\alpha\in q(L_v+u_0)$ by Corollary \ref{neig}. Moreover, one can shrink $U_v$ if necessary such that $(1-\eta) u_0 \in L_v$ for all $\eta\in U_v$. Let $$h_0= \prod_{v\in T_1} [o_v^\times : U_v]  \ \ \ \text{and} \ \ \ t_1=t_0^{h_0}  .$$ Then $t_1\in U_v$ for all $v\in T_1$.

Fix the representative $L+u_0, \ L_1+u_1, \ \cdots, \ L_s+u_s$ of all classes in $spn(L+u_0)$. The strong approximation for the spin group off $S=\{\infty_F, v_0\}$ implies $$cls(L\otimes_{\frak o_F}  \frak o_S+u_0)=cls(L_i\otimes_{\frak o_F} \frak o_S+u_i)$$ for $1\leq i\leq s$ by the condition that $V_{v_0}$ is isotropic. There is $\sigma_i\in SO(V)$ such that $$ L_i\otimes_{\frak o_F} \frak o_S = \sigma_i (L\otimes_{\frak o_F}  \frak o_S )\ \ \ \text{and} \ \ \ u_i -\sigma_i (u_0) \in \sigma_i (L\otimes_{\frak o_F}  \frak o_S )  $$ for $1\leq i\leq s$. Therefore there is a positive integer $l_i$ such that $$ t_1^{l_i} L_i \subseteq \sigma_i(L) \ \ \ \text{and} \ \ \ t_1^{l_i}(u_i-\sigma_i(u_0))\in \sigma_i(L)  $$ for $1\leq i\leq s$. Let $$h_1= \max_{1\leq i\leq s} \{l_i \} \ \ \ \text{ and } \ \ \ h=2h_1 \cdot ord_{v_0}(t_1)+1 $$ which only depends on $L$, $u_0$ and $v_0$.

If $\alpha \rightarrow spn(L+u_0)$ and $ord_{v_0}(\alpha)\geq h$ for $\alpha \in F^\times$,  then $$\beta=t_1^{-2h_1}\alpha \rightarrow gen(L+u_0) $$ by Corollary \ref{neig} for $v\in T_1$ and \cite{OM58} for $v=v_0$.  We further have that $\beta \rightarrow spn(L+u_0)$ by Corollary \ref{spn} and Lemma \ref{modify}. There is $0\leq i\leq s$ such that $\beta \in q(L_i+u_i)$. Therefore
$$\alpha =\beta \cdot t_1^{2h_1}\in t_1^{2h_1} q(L_i+u_i)=q(t_1^{h_1} L_i+ t_1^{h_1} u_i )\subseteq q(L+t_1^{h_1} u_0 ) $$ and i) follows from the fact $(t_1^{h_1}-1)u_0 \in L$ by the choice of $U_v$ for $v\in T_1$.

If $\alpha \xrightarrow{*}_T spn(L+u_0)$ and $ord_{v_0}(\alpha)\geq h$ for $\alpha \in F^\times$,  then $$T_1\cup \{v_0\} \subseteq T \ \ \ \text{ and } \ \ \ \beta=t_1^{-2h_1}\alpha \xrightarrow{*}_T  gen(L+u_0) $$ by Corollary \ref{neig} for $v\in T_1$ and \cite{OM58} for $v=v_0$. By Corollary \ref{spn} and Lemma \ref{modify}, one has $\beta \xrightarrow{*}_T spn(L+u_0)$. There is $0\leq i\leq s$ and $x\in L_i+u_i$ with $x\in^* (L_i)_v$ for $v\not\in T$ such that $q(x)=\beta$. Therefore one has $\alpha=q(t_1^{h_1} x)$ with $$t_1^{h_1} x \in \sigma_i (L+u_0) \ \ \ \text{and} \ \ \ t_1^{h_1} x \in^* \sigma_i L_v $$ for all $v\not\in T$ and ii) follows. \end{proof}
\bigskip

\section{Strong approximation with Brauer-Manin obstruction}

Let $X$ be the affine variety over $F$ defined by the equation
  \begin{equation} \label{model} q(x_1, \cdots, x_n)=p(t) \end{equation}  where $q(x_1, \cdots, x_n)$ is a non-degenerate quadratic form over $F$ with $n\geq 3$ and $p(t)$ is a non-zero polynomial over $F$. Let $f: \widetilde{X}\rightarrow X$ be a resolution of singularities for $X$, i.e. $\widetilde{X}$ is a smooth, geometrically integral $F$-variety and $f$ is a proper and birational $F$-morphism such that $$f: \  f^{-1}(X_{smooth})\xrightarrow{\cong}  X_{smooth}$$ is isomorphism, where $X_{smooth}$ is the smooth locus of $X$.

  \begin{definition} A sign $s_v$ for a real place $v$ of $F$ is defined to be $1$ or $-1$. An element $\alpha \in F^\times $ is called to be with sign $s_v$ at $v$ if $s_v \cdot \alpha>0 $ in $F_v$.
  \end{definition}

 The following lemma is well-known and is regarded as the generalized Dirichlet arithmetic progression theorem over number fields, which is the immediate consequence of the Chebatorev density theorem and the class field theory.

\begin{lem}\label{gdap}
  Let $T$ be a finite set of non-archimedean primes of $F$ and $\alpha_v\in F_v$ for each $v\in T$. Fix a sign $s_v$ for each real place $v\in \infty_F$.  For any $\epsilon >0$, there are $\alpha \in F^\times$ and a non-archimedean prime $v_0$ outside $T$ of degree one such that

1) $|\alpha-\alpha_v|_v < \epsilon $ for all $v\in T$.

2) $\alpha$ is unit at any non-archimedean prime $v \not\in T \cup \{v_0\}$ and $ord_{v_0}(\alpha)=1$.

3) $s_v \cdot \alpha >0 $ in $F_v$.

\end{lem}

\begin{proof} Without loss of generality, we can assume that $a_v\neq 0$ for $v\in T$ by replacing $a_v$ with $a_v'\in F_v^\times$ and $|a_v'|_v<\epsilon$ if necessary. Let $\Bbb R^+$ be the set of all positive reals. Choose a positive integer $c_v$ such that $$|\pi_v^{c_v}|_v <  |a_v|_v^{-1} \epsilon  $$ for any $v\in T$. By the class field theory, there is an abelian extension $E/F$ such that
$$ \Bbb I_F/(F^\times \cdot \prod_{v\in \Omega_F} U_v) \cong \Gal(E/F) $$ by Artin reciprocity law, where
$$ U_v = \begin{cases} \frak o_v^\times \ \ \ & \text{$v\not\in T\cup \infty_F$}  \\
1+\pi_v^{c_v} \frak o_v \ \ \ & \text{$v\in T$} \\
F_v^\times \ \ \ & \text{$v$ complex} \\
\Bbb R^+ \ \ \ & \text{$v$ real.} \end{cases} $$  Let
$$ i_v = \begin{cases} 1 \ \ \ & \text{$v\not\in T\cup \infty_F$}  \\
a_v^{-1} \ \ \ & \text{$v\in T$} \\
1 \ \ \ & \text{$v$ complex} \\
s_v \ \ \ & \text{$v$ real.} \end{cases} $$ Then $(i_v)_{v\in \Omega_F}\in \Bbb I_F$ and gives an element $\sigma$ in $Gal(E/F)$. By the Chebatorev density theorem, there is a finite prime $v_0$ of degree one outside $T$ such that the Frobenius of $v_0$ is $\sigma$. Therefore one has
$$ (j_v)_{v\in \Omega_F} \cdot (i_v)_{v\in \Omega_F}^{-1} \in F^\times \cdot \prod_{v\in \Omega_F} U_v $$ where
$$ j_v = \begin{cases}   \pi_{v_0}  \ \ \ & \text{$v=v_0$} \\
1 \ \ \ & \text{otherwise } \end{cases} $$ for all $v\in \Omega_F$. There is $\alpha\in F^\times$ such that
$$ \alpha^{-1} \cdot (j_v)_{v\in \Omega_F} \cdot (i_v)_{v\in \Omega_F}^{-1} \in \prod_{v\in \Omega_F} U_v $$ as required. \end{proof}

\begin{cor} \label{approx} Let $T$ be a finite set of non-archimedean primes of $F$ and $\alpha_v\in F_v$ for each $v\in T$. Fix a sign $s_v$ for each real place $v\in \infty_F$.  For any sufficiently small $\epsilon >0$ and any sufficiently large $C>0$, there are $\alpha \in F$ such that

1) $|\alpha-\alpha_v|_v < \epsilon $ for all $v\in T$.

2) $\alpha\in \frak o_v$ for all $v\not \in T\cup \infty_F$.

3) $s_v \cdot \alpha > C $ at each real place $v$ of $F$.

\end{cor}

\begin{proof} By Lemma \ref{gdap}, there are $\alpha_1$ and $\alpha_2$ in $F^\times$ such that
$$ \begin{cases} |\alpha_1-\alpha_v|_v < \epsilon  \ \ \text{and} \ \ |\alpha_2|_v<\epsilon \ \ \ & \text{for all $v\in T$} \\
  s_v \cdot \alpha_1 >0  \ \ \text{and} \ \ s_v \cdot \alpha_2 > 0 \ \ \ & \text{for all real places $v$} \\
  \alpha_1\in \frak o_v \ \ \text{and} \ \ \alpha_2 \in \frak o_v \ \ \ & \text{for all $v \not\in T\cup \infty_F$.}  \end{cases} $$

Let $k$ be a positive integer such that $s_v \cdot (\alpha_1+k \alpha_2) >C$ for all real places $v$ of $F$. Then $\alpha= \alpha_1+k \alpha_2$ is as required. \end{proof}

The main result of this section is the following theorem.

\begin{thm} \label{mainstr} If $\widetilde{X} (F_{\infty_F})$ is not compact, then strong approximation with Brauer-Manin obstruction off $\infty_F$ holds for $\widetilde{X}$.
\end{thm}
\begin{proof} By Proposition 3.4 and Theorem 6.4 in \cite{CTX1}, one only needs to consider the case that $F$ is totally real and $q$ is definite over $F_v$ for all $v\in \infty_F$. Let $c$ be the leading coefficient of $p(t)$.

Fix $\xi\in Br(\widetilde{X})$ such that $\xi$ gives the generator of $Br(\widetilde{X})/Br(F)$ by Proposition 5.6 in \cite{CTX1}. Let $V$ be a quadratic space over $F$ defined by $q$. If $W$ is an open subset of $\widetilde{X}(\Bbb A_F)$ such that
$$\widetilde{X}(\Bbb A_F)^{Br(\widetilde{X})} \cap W \neq \emptyset ,$$  there is a finite subset $T$ of $\Omega_F$ containing $\infty_F$ and all dyadic primes in $F$ such that

1).  $T$ contains all denominators of coefficients of $q(x_1, \cdots, x_n)$ and $p(t)$.

2). (\ref{model}) defines an $\frak o_T$-model of $X$ and a proper morphism of $\frak o_T$-model $f: \bf {\widetilde{X}} \rightarrow \bf{X} $ extends the morphism $f: \widetilde{X}\rightarrow X$.

3). $\xi$ takes the trivial value over ${\bf {\widetilde{X}}}(\frak o_v)$ for all $v\not\in T$.

4). there is an non-empty open subset $U_v$ of $X(F_v)_{cent}$ for each $v\in \Omega_F$ satisfying
$$ f^{-1}(\prod_{v\in \Omega_F} U_v) \subseteq W   \ \ \  \text{and} \ \ \ \widetilde{X}(\Bbb A_F)^{Br(\widetilde{X})} \cap f^{-1}(\prod_{v\in \Omega_F} U_v) \neq \emptyset . $$ Moreover, one can choose $U_v$ as follows
$$ U_v= \{(x_v,t_v)\in X(F_v) : \ x_v\in L_v+u_0  \ \ \text{and} \ \ |t_v-b_v|_v<\epsilon \}  $$
for all $v\in T\setminus \infty_F$ with sufficiently small $\epsilon>0$, where $L$ is a lattice in $V$ and $u_0\in V$ and $b_v\in F_v$ with $p(b_v)\neq 0$  and $u_0\in L_v$ for all $v\not\in T$ by Chinese Remainder Theorem and
$$  U_v ={\bf X}(\frak o_v) \cap X(F_v)_{cent}= \{(x_v,t_v)\in X(F_v)_{cent} : \ x_v \in L_v  \ \ \text{and} \ \ t_v\in \frak o_v \} $$
for all $v\not \in T$ by Lemma 8.1 and Lemma 8.3 in \cite{CTX1}, where $X(F_v)_{cent}$ is the closure of $X_{smooth}(F_v)$ in $X(F_v)$. By increasing $T$ if necessary, we also can assume that $ord_v(det(L))=0$ for all $v\not\in T$.

Let $$ P\in  \widetilde{X}(\Bbb A_F)^{Br(\widetilde{X})} \cap f^{-1}(\prod_{v\in \Omega_F} U_v) \ \ \ \text{and} \ \ \ f(P)=(a_v, b_v)_{v\in \Omega_F} \in X(\Bbb A_F) \cap \prod_{v\in \Omega_F} X(F_v)_{cent}.  $$ By the Chebatorev density theorem, there is $v_1\not \in T$ such that $p(t)$ has a root in $F_{v_1}$. Choose $t_{v_1}\in \frak o_{v_1}$ such that $p(t_{v_1})\neq 0$ and $ord_{v_1}(p(t_{v_1}))>h$ where $h=h(L,u_0, T, v_1)$ is the positive number given in Proposition \ref{arithiso}.

If $deg(p(t))$ is odd, one can choose the sign $s_v$ for all real places $v\in \infty_F$ as follows
$$ s_v = \begin{cases} 1 \ \ \ & \text{if $c\cdot q$ is positive definite at $v$} \\
-1 \ \ \ & \text{otherwise} \end{cases} $$ and the sufficiently large constant $C>0$ such that $p(t)$ and $c\cdot t$ have the same sign in $F_v$ if $|t|_v >C$ for all $v\in \infty_F$.
For sufficiently small $\epsilon >0$, there is $t_0\in F$ such that
$$ \begin{cases} |t_0-b_v|_v<\epsilon \ \ \ &  \text{$v\in T\setminus \infty_F$} \\
  |t_0-t_{v_1}|_{v_1}<\epsilon \ \ \ & \text{$v=v_1$} \\
   t_0\in \frak o_v \ \ \ & \text{$v\not\in T\cup \{v_1\}$} \\
   s_v \cdot t_0 >C \ \ \ & \text{$v\in \infty_F$} \end{cases} $$
 by Corollary \ref{approx}. By Lemma 4.2 in \cite{CTX1} for $v\not\in T$ and the inverse function theorem (see Theorem 3.2 in \cite{PR}) for $v\in T\setminus \infty_F$ and the choice of sign function $s_v$ for $v\in \infty_F$, one obtains $$p(t_0) \rightarrow gen(L+u_0). $$

Otherwise $deg(p(t))$ is even. Since $X(F_{\infty_F})$ is not compact, there is $v_0\in \infty_F$ such that $X(F_{v_0})$ is not compact. Choose a sufficiently large constant $C>0$ such that $c\cdot p(t)>0$ in $F_{v_0}$ if $|t|_{v_0}>C$.  Choose $v_2\not \in T\cup \{v_1\} $ and $t_{v_2}\in \frak o_{v_2}$ with $t_{v_2}\neq 0$ such that $$|t_{v_2}|_{v_2} < C^{-1} |t_{v_1}|_{v_1}^{-1} \prod_{v\in T\setminus \{v_0\}} (|b_{v}|_{v}+1)^{-1} .$$
Applying Chinese Remainder Theorem, one obtains $t_0\in F$ such that
$$ \begin{cases} |t_0-b_v|_v<\epsilon \ \ \ &  \text{$v\in T\setminus \{v_0\}$} \\
  |t_0-t_v|_v<\epsilon \ \ \ & \text{$v\in \{v_1, v_2 \}$} \\
   t_0\in \frak o_v \ \ \ & \text{$v\not\in T\cup \{v_1, v_2 \}$.}  \end{cases} $$
By the product formula, one has
$$ 1 \leq  |t_0|_{v_0}\cdot |t_0|_{v_1} \cdot |t_0|_{v_2}\cdot \prod_{v\in T\setminus \{v_0\}} |t_0|_v < |t_0|_{v_0}\cdot C^{-1}. $$ This implies that $c\cdot p(t_0)>0$ in $F_{v_0}$. Since $X(F_{v_0})$ is not compact, one concludes that $c\cdot q$ is positive definite at $v_0$. Therefore one has $$p(t_0) \rightarrow gen(L+u_0)$$ by Lemma 4.2 in \cite{CTX1} for $v\not\in T$ and the inverse function theorem (see Theorem 3.2 in \cite{PR}) for $v\in T\setminus \{ v_0 \}$ and the above property at $v_0$.

Moreover, one has
$$p(t_0) \rightarrow spn(L+u_0) $$ by Proposition 5.6 (c) and Lemma 4.4 in \cite{CTX1} and the functoriality and continuality of Brauer-Manin pairing (see (5.3) in \cite{Sko}) and the proof Proposition 7.3 in \cite{CTX}. Therefore
$$ p(t_0) \rightarrow cls (L+u_0) $$ by Proposition \ref{arithiso} i). Since all points in the fiber of $X$ over $t=t_0$ are smooth, the morphism $f$ induces $$ \widetilde{X}_{t_0} \cong X_{t_0} $$ defined by  $q(x_1, \cdots, x_n)=p(t_0)$. This implies that
$$ \widetilde{X}_{t_0}(F) \cap f^{-1}(\prod_{v\in \Omega_F} U_v) \neq \emptyset $$
and the proof is complete.  \end{proof}

\begin{rem} The $D$ defined by (2.2) in \cite{Wat} is equal to $p(t)$ up to a constant in $\Bbb Q^\times$. When $n\geq 4$, one has $Br( \widetilde{X})=Br(F)$ and strong approximation off $\infty_F$ holds for $ \widetilde{X}$ or by the same line of the proof of Theorem \ref{mainstr}. Therefore the local-global principle holds for (\ref{watson}) with $n\geq 4$, which is equivalent to Theorem 1 and Theorem 2 in \cite{Wat}. The most interesting case is that $n=3$. The explicit computation for $Br(\widetilde{X})/Br(F)$ is given by Proposition 5.6 in \cite{CTX1}. The condition of Theorem 3 in \cite{Wat} is that $p(t)$ has a root over $\Bbb Q_p$ for almost all primes $p$. This condition implies that $Br(\widetilde{X})=Br(F)$ by Remark 6.6 in \cite{CTX1}. Therefore strong approximation off $\infty_F$ holds for $ \widetilde{X}$ with $n=3$ under such a condition. This implies that the local-global principle holds for (\ref{watson}) with $n=3$ under such a condition, which is equivalent to Theorem 3 in \cite{Wat}.
\end{rem}

\bigskip

\section{Representations of definite quadratic polynomials}

By using the circle method as pointed out in \cite{Wat60}, one can prove the representability of definite quadratic polynomials with more than three variables. On the other hand, primitive representations of definite quadratic forms of more than four variables with congruence conditions is proved in Theorem 2.1 of \cite{HJ} by an arithmetic method. In this section, we will point out that this two statements are equivalent to each other and extend them over number fields based on the results in \S \ref{qde}.

\begin{definition} Let $V$ be a non-degenerated quadratic space over $F$ and $w\in \infty_F$ such that $V_w$ is positive definite and $T$ be a finite subset containing all dyadic primes with $T\subset (\Omega_F\setminus \infty_F)$.

\underline{Statement (CC) with respect to $T$}: \ If $L$ is a lattice in $V$  satisfying that $L_v$ is unimodular for all $v\not\in T$ and $s$ is a positive integer, there is a constant $c=c(L,T,s)$ depending only on $L$, $T$ and $s$ such that for any $\alpha \in F$ satisfying

1) $\alpha \rightarrow gen(L)$

2) $x_v\in L_v$ with $q(x_v)= \alpha$ for $v\in T$ and $x_v \in^* L_v$ when $V_v$ is anisotropic with $v\in T$

3) $\alpha > c$ at $w$
\newline
one has $x\in L$ with $q(x)=\alpha$ such that
$$ \begin{cases}  x\equiv x_v \mod \pi_v^sL_v \ \ \ &  \text{for $v\in T$} \\
 x\in^*L_v  \ \ \ & \text{for $v\not\in T$}
\end{cases} $$

\bigskip

\underline{Statement (LT) with respect to $T$}: \ If $L$ is a lattice in $V$  and $u_0\in V$ satisfying that $L_v$ is unimodular and $u_0\in L_v$ for all $v\not\in T$, there is a constant $c=c(L,T,u_0)$ depending only on $L$, $T$ and $u_0$ such that for any $\alpha \in F$ satisfying

1) $\alpha \rightarrow gen(L+u_0)$

2) $\alpha \xrightarrow{*}  L_v$ when $V_v$ is anisotropic and $u_0\in L_v$ for $v\in T$

3) $\alpha > c$ at $w$
\newline
one has $x\in L+u_0$ with $ x\in^*L_v$ for $v\not\in T$ such that $q(x)=\alpha$.
\end{definition}

\begin{prop} Let $V$ be a non-degenerated quadratic space over $F$ and $w\in \infty_F$ such that $V_w$ is positive definite and $T$ be a finite subset with $T\subset (\Omega_F\setminus \infty_F)$. Then Statement (CC) with respect to $T$ holds if and only if Statement (LT) with respect to $T$ holds.
\end{prop}

\begin{proof} $(\Rightarrow)$ Let $K= L+\frak o_F u_0$ be a lattice in $V$. Since $\alpha \rightarrow gen(L+u_0)$, there is $x_v\in L_v$ such that $q(x_v+u_0)=\alpha$ for $v\in T$. Choose a positive integer $s$ so large such that $\pi_v^s u_0 \in L_v$ for all $v\in T$. By Statement (CC) with respect to $T$, there is $x\in K$ with $q(x)=\alpha$ such that $ x\equiv x_v+u_0  \mod \pi_v^sK_v$ for $v\in T$ and $x\in^*K_v$ for $v\not\in T$. Since $\pi_v^sK_v \subseteq L_v$ for all $v\in T$ and $K_v=L_v$ for all $v\not\in T$, one concludes that $x\in L+u_0$ as required.

$(\Leftarrow)$ Conversely, define a lattice $K$ in $V$
$$  \begin{cases}  K_v= \pi_v^sL_v \ \ \ & v\in T \\
 K_v= L_v  \ \ \ & v\not\in T
\end{cases} $$ and $u_0\in V$ satisfying $$ \begin{cases} u_0-x_v \in \pi_v^sL_v \ \ \ & v\in T \\ u_0\in L_v \ \ \ & v\not\in T \end{cases}$$ by Chinese Remainder Theorem. By Statement (LT) with respect to $T$, there is $x\in K+u_0$ with $ x\in^*K_v$ for $v\not\in T$ such that $q(x)=\alpha$. Then
$$ \begin{cases}  x\equiv u_0 \equiv x_v \mod \pi_v^sL_v \ \ \ & \text{for $v\in T$} \\
 x\in^*L_v  \ \ \ & \text{for $v\not\in T$}
\end{cases} $$ by $K_v=L_v$ for $v\not \in T$.  \end{proof}

By applying Lemma \ref{gdap}, one can obtain the following approximation result which was first proved by Kitaoka for usual lattices over $\Bbb Z$ in \cite{Ki} (see also Lemma 1.6 in \cite{HKK} and Theorem 6.2.1 in \cite{Ki1}).

\begin{lem} \label{cs} Suppose $V$ is a non-degenerated quadratic space over $F$ such that $dim(V)\geq 3$ or $dim(V)=2$ with $-det(V)\not\in (F^\times)^2$. Let $L$ be a lattice in $V$ and $u_0\in V$ and $T$ be a finite set of non-archimedean primes of $F$ such that $T$ contains all dyadic primes and $L_v$ is unimodular  and $u_0\in L_v$ for $v\not \in T$. Take $u_v\in L_v+u_0$ for each $v\in T$. For any $\epsilon >0$, there is $u\in L+u_0$ such that

1) $|u-u_v|_v < \epsilon $ for all $v\in T$

2) there is a finite prime $v_0$ outside $T$ satisfying that $q(u)$ is unit at any non-archimedean prime $v \not\in T \cup \{v_0\}$ and $ord_{v_0}(q(u))=1$.

\end{lem}

\begin{proof} Step 1.  If $dim(V)=2$ and $-det(V)\not\in (F^\times)^2$, there is $a\in F^\times$ such that $$ \phi: \ \ (V/F, \ a\cdot q)\cong (E/F,\  N_{E/F})$$ as quadratic spaces, where $E=F(\sqrt{-det(V)})$ and $N_{E/F}$ is the usual norm map. Let
$$ T_1= \{ v\in \Omega_F \setminus (T\cup \infty_F): \ \ \phi(L)_v \neq {o_{E}}_v \} $$
where $o_E$ is the integral closure of $\frak o_F$ inside $E/F$ and $u_v\in L_v$ such that $q(u_v)\in o_v^\times$ for all $v\in T_1$ since $L_v$ is unimodular for $v\not\in T$.

For sufficiently small $\epsilon >0$, one obtains $\lambda \in E$ and a non-archimedean prime $w_0$ of $E$ of degree one such that
$|\lambda-\phi(u_v) |_w < \epsilon $ for all $w$ above $T\cup T_1$ and $\lambda$ is unit at any non-archimedean prime $w$ which is not above $T\cup T_1$ but $ord_{w_0}(\lambda)=1$ by applying Lemma \ref{gdap} for the set of primes of $E$ above $T\cup T_1$.

We claim that $u=\phi^{-1}(\lambda)$ is as required. Indeed, $u\in L_v+u_0$ for all $v\in \Omega_F$ by our construction. This implies that $u\in L+u_0$. For any finite prime $v$ with $v\not\in T\cup T_1$, one has both $q(L_v)$ and $N_{E/F}({\frak o_{E}}_v)$ contain the units of $\frak o_v^\times$. This implies that $ord_v(a)=0$ for $v\not\in T\cup T_1$. Therefore
$$ ord_{v_0}(q(u))=ord_{v_0}(a\cdot q(u))=ord_{v_0}(N_{E/F}(\lambda))=1 $$ where $v_0$ is the prime in $F$ below $w_0$ and
$$ ord_v(q(u))=ord_v(a\cdot q(u))=ord_v(N_{E/F}(\lambda))= 0 $$ for all finite primes $v\not\in (T\cup T_1\cup \{v_0\})$. Since $|u-u_v|_v<\epsilon$ with sufficient small $\epsilon$ for $v\in T_1$, one has $ord_v(q(u))=ord_v(q(u_v))=0$ for $v\in T_1$ by the choice of $u_v$ for $v\in T_1$. The claim follows. \medskip

Step 2.  If $dim(V)\geq 3$, we set $M=L+ \frak o_F u_0$. Choose a finite $v_1\not\in T$. Then $$M_{v_1}=L_{v_1}, \ \ \ \frak o_{v_1}^\times \subset q(M_{v_1}) \ \ \ \text{and} \ \ \ [\frak o_{v_1}^\times: (\frak o_{v_1}^\times)^2]=2. $$  There is a unimodular sublattice $K\subset M_{v_1}$ satisfying $$rank(K)=2 \ \ \ \text{and} \ \ \ -det(K) \not\in (F_{v_1}^\times)^2 . $$ Write $K=\frak o_{v_1}u_{v_1} \perp \frak o_{v_1}\eta_{v_1}$.
For sufficiently small $\epsilon >0$, there is $\xi\in M$ such that $$ |\xi-u_{v_1}|_{v_1}<\epsilon \ \ \ \text{and} \ \ \ |\xi-u_v|_v <\epsilon  $$ for $v\in T$ by Chinese Remainder Theorem. Let $$ T_1= \{ v\in \Omega_F\setminus (T\cup \{v_1\} \cup \infty_F): \ \ ord_v(q(\xi))\neq 0 \} . $$ Choose $\eta_v\in M_v$ such that $ord_v(q(\eta_v))=0$ for each $v\in T_1$. For sufficiently small $\epsilon >0$, there is $\eta\in M$ such that $$ |\eta-\eta_{v_1}|_{v_1}<\epsilon \ \ \ \text{and} \ \ \ |\eta-\eta_v|_v <\epsilon  $$ for $v\in T_1$ and $\xi$ and $\eta$ are linearly independent over $F$ by Chinese Remainder Theorem. Let $$N=\frak o_F \xi+ \frak o_F \eta \ \ \text{and} \ \  T_2=\{ v\in \Omega_F\setminus (T\cup \{v_1\} \cup T_1\cup \infty_F): \ \text{$N_v$ is not unimodular} \}  $$ and $$ \delta_v = \begin{cases} \xi \ \ \ & v\in T\cup T_2 \\
\eta \ \ \ & v\in T_1 .  \end{cases} $$ Since $$-det(N)(F_{v_1}^\times)^2 = -det(K)(F_{v_1}^\times)^2 \neq (F_{v_1}^\times)^2$$ for sufficiently small $\epsilon$ by our choice of $v_1$ and approximation, one can apply Step 1 to $N$ with given vectors $\delta_v$ for $v\in T\cup T_1\cup T_2$. For sufficiently small $\epsilon >0$, there are $u\in N$ such that $|u-\delta_v|_v<\epsilon$ for all $v\in T\cup T_1\cup T_2$ and $ord_v(q(u))=0$ for all primes $v\not\in T\cup T_1\cup T_2$ except one prime $v_0$ and $ord_{v_0}(q(u))=1$. By the approximation on $T$, one concludes that $u\in L_v+u_0$ for all $v\in \Omega_F$. Therefore $u\in L+u_0$. By approximation on $T_1\cup T_2$, one also has that $ord_v(q(u))=0$ for all $v\in T_1\cup T_2$. The proof is complete. \end{proof}

We extend \S 2 in \cite{HJ} over a number field.

\begin{lem}\label{locorth} Suppose $V_v$ is a non-degenerated quadratic space over $F_v$ such that $dim(V_v)\geq 3$ or $dim(V_v)=2$ with $-det(V_v)\not\in (F_v^\times)^2$ for $v<\infty_F$. Let $L_v$ be a lattice in $V_v$ and $s$ be a positive integer. For any $x\in V_v$ with $x \neq 0$,  there is an anisotropic vector $y\in V_v$  such that $x-y$ is anisotropic and $$x-y \in (F_vy)^{\perp}\cap \pi_v^s L_v. $$
\end{lem}

\begin{proof} If $x$ is anisotropic, one can choose an anisotropic vector $\xi$ such that $$V_v=F_vx\perp F_v \xi\perp W$$ for some regular subspace $W$. Let $t$ be a sufficiently large positive integer such that $$\pi_v^t \xi\in \pi_v^s L_v, \ \  \ q(\xi)q(x)^{-1} \pi_v^{2t} x \in \pi_v^s L_v  \ \ \ \text{ and } \ \ \ q(x)+\pi_v^{2t} q(\xi) \neq 0 . $$
Let $$ y=(1+q(\xi)q(x)^{-1}\pi_v^{2t})^{-1} (x+\pi_v^t \xi)  $$ and $y$ is anisotropic by the choice of $t$.  Then
$$ x-y = (1+q(\xi)q(x)^{-1}\pi_v^{2t})^{-1}( q(\xi)q(x)^{-1}\pi_v^{2t} x - \pi_v^t \xi) \in \pi_v^s L_v$$ and
$ \langle x-y, y \rangle = 0$. Moreover,
$$ q(x-y)= (1+q(\xi)q(x)^{-1}\pi_v^{2t})^{-1} \cdot \pi_v^{2t} \cdot q(\xi) \neq 0 $$ as required.

If $x$ is isotropic, there is $w\in V_v$ such that $\langle x, w \rangle=1$ and $q(w)=0$. By the assumption on $V_v$, there are an anisotropic vector $h\in V_v$ and a regular subspace $W$ of $V_v$ such that $$V_v=(F_v x+F_v w)\perp F_v h \perp W . $$ Let $t$ be a sufficiently large positive integer such that $$\pi_v^t h \in \pi_v^s L_v, \ \ \ q(h) \pi_v^{2t} w \in \pi_v^s L_v  \ \ \ \text{ and } \ \  \   y = x-\pi_v^{2t} q(h) w+ \pi_v^t h . $$
Then $y$ is anisotropic and $\langle x-y, y \rangle = 0 $ and $x-y\in \pi_v^s L_v$ as required.
\end{proof}

Modifying the above result, one can extend the statement to $x=0$.

\begin{cor}\label{loclatt}  Suppose $V_v$ is a non-degenerated quadratic space over $F_v$ such that $dim(V_v)\geq 3$ or $dim(V_v)=2$ with $-det(V_v)\not\in (F_v^\times)^2$ for $v<\infty_F$. Let $L_v$ be a lattice in $V_v$ and $s$ be a positive integer.
 If $x\in V_v$ with $x \neq 0$,  there is an anisotropic vector $y\in V_v$ and a sublattice $K_v$ in $(F_vy)^{\perp}\cap \pi_v^s L_v$ such that $x-y \in \pi_v^s L_v $ and $0\neq q(x)-q(y)$ is represented by $K_v$ primitively. Moreover, if $V_v$ is isotropic, this is also true for $x=0$.
\end{cor}

\begin{proof} By lemma \ref{locorth}, one only needs to choose a sublattice $K$ of $(F_vy)^{\perp}\cap \pi_v^s L_v$ with the maximal rank satisfying $x-y \in^* K$. In order to extend this statement to $x=0$ for isotropic $V_v$, one can choose an isotropic vector $x'\in \pi_v^s L_v$ with $x'\neq 0$ and apply the previous result for $x'$. The result follows from $q(x')=q(0)=0$.
\end{proof}

\begin{definition} \label{ass} Suppose $V_v$ is a non-degenerated quadratic space over $F_v$ for $v<\infty_F$. Let $L_v$ be a lattice in $V_v$ and $s$ be a positive integer. For any $x\in V_v$, we call the anisotropic vector $y$ and the sublattice $K_v$ in $(F_vy)^{\perp}\cap \pi_v^s L_v$ satisfying the property of Corollary \ref{loclatt} the associated vector and the associated lattice of $x$ with respect to $L_v$ and $s$ respectively.
\end{definition}

It should be pointed out that Corollary \ref{loclatt} can not extended to $x=0$ if $V_v$ is anisotropic since $-q(y)$ can not be represented by the space $(F_vy)^{\perp}$ by the anisotropic assumption.

\begin{lem}\label{witt} Suppose $V_v$ is a non-degenerated quadratic space over $F_v$ for $v<\infty_F$. Let $L_v$ be a lattice in $V_v$ and $x\in V_v$. If $y$ is sufficiently close to $x$, there is $\sigma\in O(L_v)$ such that $F_v\sigma x =F_v y$. \end{lem}
\begin{proof} Without loss of generality, one can assume $q(L_v)\subseteq \frak o_v$ by scaling. Let $j\in \Bbb Z$ such that $\pi_v^j x$ is primitive in $L_v^\sharp$. If $y$ is so close to $x$ such that $\pi_v^j y$ is also primitive in $L_v^\sharp$ and $q(y)=\epsilon^2 q(x)$ with $\epsilon\in \frak o_v^\times$ and $\pi_v^j(\epsilon^{-1}y-x)\in L_v$, then there is $\tau\in O(V_v)$ such that $\tau(x)=\epsilon^{-1}y$ by Witt theorem. Applying Lemma 2.1 in \cite{SX}, one obtains $\sigma\in O(L_v)$ such that $\frak o_v \sigma x = \frak o_v y$ as required. \end{proof}

One needs to modify the local cover in \cite{HJ} to the following corollary.

\begin{cor}\label{finilatt}  Suppose $V_v$ is a non-degenerated quadratic space over $F_v$ such that $dim(V_v)\geq 3$ or $dim(V_v)=2$ with $-det(V_v)\not\in (F_v^\times)^2$ for $v<\infty_F$. Let $L_v$ be a lattice in $V_v$ and $s$ be a positive integer.

 If $P$ is an open compact subset of $V_v$ such that $0\not\in P$ if $V_v$ is anisotropic, there are $\delta>0$ and a finite subset $B\subset P$ and a finite set $\frak B$ of the associated vectors of $B$ with respect to $L_v$ and $s$ satisfying the following property:

 For any $x\in P$,  there is $b\in B$ with the associated vector $\beta_b\in \frak B$ such that $x-b\in \pi_v^s L_v$ and for any $|\xi-\beta_b|_v<\delta$ one has  $\xi- x \in \pi_v^s L_v$ and there is a sublattice $K_v$ of $\pi_v^s L_v\cap (F_v \xi)^\perp$ satisfying that whenever $|u-\beta_b|_v<\delta$ for $u\in V_v$ one has $q(x)-q(u)$ is represented by $K_v$ primitively.
\end{cor}

\begin{proof} Let $\delta_0>0$ such that for any $|x-y|_v < \delta_0$ with $x, y\in V_v$ one has $x-y \in \pi_v^s L_v$.

For any $x\in P$, one fixes the associated vector $\beta_x$ and the associated lattice  $K_v(\beta_x)$ of $x$ with respect to $L_v$ and $s$ by Corollary \ref{loclatt}.  Choose $0<\delta(x,\beta_x)< \delta_0$ so small such that

1)  $$U(x, \delta(x,\beta_x))=\{ z\in V_v: \ |z-x|_v<\delta(x,\beta_x)\} \subseteq P $$ and  $z-x\in \pi_v^s L_v$  for $z\in U(x, \delta(x,\beta_x))$.

2) if $$|\xi-x|_v<\delta(x,\beta_x) \ \ \ \text{ and } \ \ \ |\eta-\beta_x|_v<\delta(x,\beta_x) $$ for $\xi, \eta\in V_v$, one has $q(\xi)-q(\eta)$ is equal to $q(x)-q(\beta_x)$ up to a square of unit in $\frak o_{v}$.

 3) if $|z-\beta_x|_v <\delta(x,\beta_x)$, there is $\sigma\in O(L_v)$ such that $F_v \sigma \beta_x=F_v z$ by Lemma \ref{witt}.

 By compactness of $P$,  there is a finite subset $B$ of $P$ such that $\{ U(b, \delta(b,\beta_b) \}_{b\in B}$ is a cover of $P$.  Define $$\delta= \min_{b\in B} \{ \delta(b, \beta_b)\} . $$

For any $x\in P$, there is $b\in B$ such that $x\in U(b, \delta(b, \beta_b))$. If $|\xi-\beta_b|<\delta\leq \delta_0$, one has $$\xi-\beta_b\in \pi_v^s L_v \ \ \ \text{ and } \ \ \ x-b\in \pi_v^s L_v $$ by the choice of $\delta_0$. By corollary \ref{loclatt}, one has $b-\beta_b\in \pi_v^s L_v$. Therefore $\xi- x \in \pi_v^s L_v$.

By corollary \ref{loclatt}, there is a sub-lattice $K(\beta_b)$ of maximal rank in $(F_v \beta_b)^\perp \cap \pi_v^s L_v$ such that $q(b)-q(\beta_b)$ is represented by $K(\beta_b)$ primitively. By 3), there is $\sigma\in O(L_v)$ such that $F_v\sigma \beta_b = F_v \xi$. Set $$K_v= \sigma K(\beta_b) \subseteq \pi_v^s L_v\cap (F_v \xi)^\perp . $$
For any $|u-\beta_b|<\delta$ with $u\in V_v$, one has $q(x)-q(u)$ is equal to $q(b)-q(\beta_b)$ up to a square of unit in $\frak o_{v}$ by 2). Therefore $q(x)-q(u)$ is represented by $K_v=\sigma K(\beta_b)$ primitively.
\end{proof}

When $V_v$ is anisotropic, one can only apply Corollary \ref{loclatt} to $L_v\setminus \{0\}$ which is not compact. This fact produces the counter-example (9.1) in \cite{Cas} where the local-global principle fails.  In order to restrict to some compact subset of $L_v\setminus \{0\}$, one can consider, for example, the set of the primitive vectors of $L_v$.

\begin{thm} \label{defrep} Let $V$ be a non-degenerated quadratic space over $F$ and $w\in \infty_F$ such that $V_w$ is positive definite and $T$ be a finite set of non-archimedean primes containing all dyadic primes. If $dim(V)=4$, then Statement (LT) with respect to $T$ holds.
\end{thm}

\begin{proof}  By Theorem 8.3 in \cite{CTX}, one only needs to consider that $V_v$ is definite for all $v\in \infty_F$.
We first prove the following claim.
\begin{claim} \label{claim} For any non-archimedean prime $v_0\not\in T$, there is a positive number $$C_1=C_1(L,T,u_0,v_0)$$ such that whenever $\alpha\in F$ satisfies $\alpha > C_1$ at $w$ and $ord_{v_0}(\alpha)=0$ and $\alpha \rightarrow gen(L+u_0)$ with
$ \alpha \xrightarrow{*} L_v $ for those $v\in T$ where $V_v$ is anisotropic and $u_0\in L_v$, then there is $x\in L+u_0$ with $x\in^* L_v$ for all $v\not\in T$ such that $q(x)=\alpha$.
\end{claim}

Indeed, we put
$$ P_v= \begin{cases}  \{ x\in^* L_v \} \ \ \ & \text{$V_v$ is anisotropic and $u_0\in L_v$} \\
L_v+u_0 \ \ \ & \text{otherwise} \end{cases} $$
for all $v\in T$ and fix an positive integer $s$. Then there are $\delta_v>0$ and a finite subset $B_v\subset P_v$ and a finite set $\frak B_v$ of the associated vectors of $B_v$ with respect to $L_v$ and $s$ satisfying the property of Corollary \ref{finilatt}. By the choice of $P_v$ and Corollary \ref{loclatt}, one has $\frak B_v \subset P_v$ for $v\in T$. Let $\frak B_{v_0}=\{x_{v_0}, y_{v_0}\} \subset L_{v_0}$ such that $q(x_{v_0})= 1$ and $q(y_{v_0})$ is a non-square unit in $\frak o_{v_0}$. Choose $\delta_{v_0}>0$ so small such that if $|x-\lambda_{v_0}|_{v_0}<\delta_{v_0}$ for $\lambda_{v_0}\in \frak B_{v_0}$ then $$q(\lambda_{v_0})q(x)^{-1} \in (\frak o_{v_0}^\times)^2 . $$ Set
$$S=T\cup \{v_0\} \ \ \ \text{and} \ \ \ \Lambda =\prod_{v\in S} \frak B_v . $$
For any $\lambda=(\lambda_v)_{v\in S}\in \Lambda$, there is $u_\lambda\in L+u_0$ and one prime $v_\lambda \not\in S$ such that $|u_\lambda -\lambda_v|_v<\delta_v$ for all $v\in S$ and $$ord_v(q(u_\lambda))= \begin{cases} 0 \ \ \ & v\not\in S\cup\{v_\lambda\} \\ 1 \ \ \ & v=v_\lambda \end{cases} $$ by Lemma \ref{cs}. There is a sublattice $K_v$ of $(F_v u_\lambda)^\perp $ satisfying the property of Corollary \ref{finilatt} for $v\in T$. Construct the sublattice $K_\lambda$ of $L\cap (F u_\lambda)^\perp$ as follows
$$(K_\lambda)_v= \begin{cases}  L_v\cap (F_v u_\lambda)^\perp \ \ \ & \text{if $v\not\in T$} \\
K_v \ \ \ & \text{if $v\in T$}
\end{cases} $$
and $(K_\lambda)_{v_0}$ is unimodular by $ord_{v_0}(q(u_\lambda))=0$.

Applying Proposition \ref{arithiso} ii) to $K_\lambda$ for primitive representation outside $\{v_0\}$ with $\lambda \in \Lambda$, one obtains a positive number $h_\lambda$ for $\lambda \in \Lambda$. Let $h=\max_{\lambda\in \Lambda} h_\lambda$ and $$ \gamma = \min_{ v\in \infty_F} \min \{ |q(x)|_v : \ x (\neq 0) \in L'+u_0' \ \ \text{with} \ \  cls(L'+u_0')\in gen(L+u_0) \} .  $$  For any $$\lambda\in \Lambda, \ \ \ a_{v_0}\in (\frak o_{v_0}/(\pi_{v_0}^{h+3}))^\times \ \ \ \text{ and } \ \ \ a_{v_\lambda}\in \frak o_{v_\lambda}/(\pi_{v_\lambda}) , $$ there is $\xi=\xi(\lambda, a_{v_0}, a_{v_\lambda})\in \frak o_F$ such that
 \begin{equation} \label{xi} \begin{cases} |\lambda_v-\xi u_\lambda |_v=|\lambda_v-u_\lambda|_v \ \ \text{and} \ \ (\xi-1)u_0 \in L_v \ \ \ & v\in T \\
ord_{v_0}(\xi-a_{v_0})>h  \ \ \text{and} \ \  ord_{v_0}(\xi-a_{v_0})\equiv 1 \mod 2  \ \ \ & v=v_0 \\
\xi\equiv a_{v_\lambda} \mod \pi_{v_\lambda} \ \ \ & v=v_\lambda \\
 |q(\xi u_\lambda)|_v <\frac{1}{2} \gamma \ \ \ & v\in \infty_F \setminus \{w\} \end{cases}  \end{equation}  by Chinese Remainder Theorem. Let $$C_1=C_1(L,T,u_0,v_0)= \max_{\xi} \{ |q(\xi u_\lambda)|_w \}  $$ where $\xi=\xi(\lambda,a_{v_0}, a_{v_\lambda})$ runs over all the above chosen values.

 Since $\alpha \rightarrow gen(L+u_0)$ with
$ \alpha \xrightarrow{*} L_v $ for those $v\in T$ where $V_v$ is anisotropic and $u_0\in L_v$, there is $y_v\in P_v$ with $q(y_v)=\alpha $ for all $v\in T$.  Then there is $b_v\in B_v$ and the associated vector $\lambda_v$ of $b_v$ with respect to $L_v$ and $s$ for $v\in T$ satisfying the property of Corollary \ref{finilatt}. Let $\lambda_{v_0}\in \frak B_{v_0}$ such that $$\alpha \cdot q(\lambda_{v_0})^{-1}\in (\frak o_{v_0}^\times)^2 $$ by the assumption for $\alpha$. Then $\lambda=(\lambda_v)_{v\in S}\in \Lambda$. For such $\lambda\in \Lambda$, one has $u_\lambda\in L+u_0$ and the prime $v_\lambda\not\in S$ and $K_\lambda$ as above with $$ord_{v_\lambda}(q(u_\lambda))=ord_{v_\lambda}(\det(K_\lambda))=1 . $$

If $ord_{v_\lambda}(\alpha)=1$, there are $x\in \frak o_{\lambda_v}^\times$ and $ \eta_{v_\lambda} \in \frak o_{\lambda_v}$ satisfying the following equation
\begin{equation} \label{equ} \frac{\alpha}{q(u_\lambda)}-\eta_{v_\lambda}^2 \equiv  \frac{-\Delta_{\lambda_v}\cdot \det(K_{\lambda})}{q(u_{\lambda})} x^2  \mod \pi_{\lambda_v} \end{equation}
where $\Delta_{\lambda_v}$ is a fixed non-square unit in $\frak o_{\lambda_v}^\times$. Choose
 $$a_{v_\lambda}\equiv \begin{cases} \eta_{v_\lambda} \ \ \ & \text{$ord_{v_\lambda}(\alpha)= 1$} \\
 1 \ \ \ & \text{otherwise} \end{cases}  \mod \pi_{v_\lambda} . $$
 By the choice of $\delta_{v_0}$, there is $\eta_{v_0}\in \frak o_{v_0}^\times $ such that $\alpha =\eta_{v_0}^2  q(u_{\lambda})$.  Choose
$$a_{v_0}\equiv \eta_{v_0} \mod \pi_{v_0}^{h+3} . $$ For the above chosen $\lambda$, $a_{v_0}$ and $a_{v_\lambda}$, we get $\xi=\xi(\lambda, a_{v_0}, a_{v_\lambda})$ as (\ref{xi}).

\medskip

Next we will show that $\alpha-q(\xi u_\lambda) \xrightarrow{*} (K_\lambda)_v$ for all $v\in \Omega_F$.

When $v\not\in T\cup \{v_\lambda\}\cup \{\infty_F\} $, the claim follows from Lemma 4.2 in \cite{CTX1}.

When $v=v_\lambda$, one has that the first Jordan component of $(K_\lambda)_{v_\lambda}$ is a unimodular lattice of rank $(dim(V)-2)$. One only needs to consider the first Jordan component of $(K_\lambda)_{v_\lambda}$ is of rank 2 and is not a hyperbolic plane. Then the discriminant of the first Jordan component is $-\Delta_{\lambda_v}$. If $\alpha$ is a unit, then $$\alpha-q(\xi u_\lambda) \xrightarrow{*} (K_\lambda)_{v_\lambda}$$ by Hensel's Lemma. If $ord_{v_\lambda}(\alpha)\geq 1$, then $ord_{v_\lambda}(\alpha-q(\xi u_\lambda))=1$ and $\alpha-q(\xi u_\lambda)$ is represented by $(K_\lambda)_{v_\lambda}$ by the equation (\ref{equ}). This implies that
$$\alpha-q(\xi u_\lambda) \xrightarrow{*} (K_\lambda)_{v_\lambda}$$ by the order consideration of the coefficients.

When $v\in T$, then $$\alpha-q(\xi u_\lambda)\xrightarrow{*} (K_\lambda)_v$$ by Corollary \ref{finilatt}.

When $v\in \infty_F\setminus \{w\} $, then $\alpha-q(\xi u_\lambda)$ and $\alpha$ have the same sign at $v$ by the choice of $\gamma$. Therefore $\alpha-q(\xi u_\lambda)$ is represented by $F_vK_\lambda$.

When $v=w$, the assumption that $\alpha > C_1$ at $w$ implies that $\alpha-q(\xi u_\lambda)$ is positive at $w$. Therefore $\alpha-q(\xi u_\lambda)$ is represented by $F_wK_\lambda$.

\medskip

Since $ord_{v_0}(\alpha-q(\xi u_\lambda))>h$ and $ord_{v_0}(\alpha-q(\xi u_\lambda))\equiv 1 \mod 2$ by (\ref{xi}), one obtains $z\in K_\lambda$ with $z\in^* (K_\lambda)_v$ for $v\neq v_0$ such that $q(z)=\alpha-q(\xi u_\lambda)$ by applying Corollary \ref{spn} and Proposition \ref{arithiso} to $K_\lambda$ and $\{ v_0 \}$. Therefore $\alpha=q(\xi u_\lambda+z)$ with $\xi u_{\lambda}+z\in L+u_0$. Since $$(K_\lambda)_v =  L_v \cap (F_v u_\lambda)^\perp $$ for $v\not \in T$, one concludes that $\xi u_{\lambda}+z\in^* L_v$ for all $v\not\in T$ and $v\neq v_0$. One also has $\xi u_{\lambda}+z\in^* L_{v_0}$ by $ord(\alpha)=ord(q(\xi u_\lambda+z))=0$. We complete the proof of Claim \ref{claim}.

\medskip

Finally we need to modify Cassels' trick (see Lemma 9.1 of Chapter 11 in \cite{Cas}) over a number field to complete the proof by using Claim \ref{claim}. Choose a non-archimedean $v_0\not \in T$ such that $d(L_{v_0}) \in (\frak o_{v_0}^\times)^2$ and $v_0$ is a principle ideal of $\frak o_F$ by the Chebatorev density theorem. Let $V^\alpha$ be the quadratic space obtained by scaling $V$ with $\alpha\in F^\times$. Fix a generator $\pi_{v_0}$ of prime ideal $v_0$ of $\frak o_F$ such that $0< \pi_{v_0} <1$ at $w$. Then $$V\cong V^{\pi_{v_0}^{2m}} \ \ \ \text{and} \ \ \ V^{\pi_{v_0}}\cong V^{{\pi_{v_0}}^{2m+1}}$$ over $F$ for all $m\in \Bbb Z$ by 66:5 Remark in \cite{OM}. Since there are only finitely many classes with the given scale and volume in a non-degenerated quadratic space by 103:4 Theorem in \cite{OM}, one concludes that the set
$$ \frak {C}= \{ cls(K+u_0): \text{$K$ is a lattice in $V^{\pi_{v_0}^m}$ with} \  \frak s(K)=\frak s(L) \ \text{and} \ \frak v(K)=\frak v(L),  \ \forall m\in \Bbb Z \} $$ is finite, where $\frak s(K)$ and $\frak s(L)$ are the scales of $K$ and $L$ respectively and $\frak v(K)$ and $\frak v(L)$ are the volumes of $K$ and $L$ respectively defined as \S 82E in \cite{OM}. Let
$$ c=c(L,T,u_0) = \max_{cls(K+u_0)\in \frak C} \{ C_1(L,T,u_0,v_0) \ \text{as Claim \ref{claim}}  \}. $$

If $\alpha \rightarrow gen(L+u_0)$ with $\alpha > c$ at $w$ and $\alpha \xrightarrow{*}  L_v$ when $V_v$ is anisotropic and $u_0\in L_v$ for $v\in T$, there is a positive integer $k$ such that $ord_{v_0}(\pi_{v_0}^{-k} \alpha)=0$ and $\pi_{v_0}^{-k} \alpha> \alpha >c$ at $w$. Since
$$ L_{v_0} = (\frak o_{v_0} e_1 + \frak o_{v_0} f_1) \perp (\frak o_{v_0} e_2 + \frak o_{v_0} f_2)$$ with $q(e_1)=q(f_1)=q(e_2)=q(f_2)=0$ and $\langle e_1,f_1 \rangle =\langle e_2, f_2 \rangle =1 $, one defines $M\subset L$ as follows
$$ M_v = \begin{cases}  (\frak o_{v_0} e_1 + \frak o_{v_0} \pi_{v_0}^k f_1) \perp (\frak o_{v_0} e_2 + \frak o_{v_0} \pi_{v_0}^k  f_2) \ \ \ & v=v_0 \\
L_v \ \ \ & v\neq v_0.  \end{cases} $$
Then $M$ is a lattice in $V^{\pi_{v_0}^{-k}}$ such that $cls(M+u_0) \in \frak C$. Applying Claim \ref{claim} for $\pi_{v_0}^{-k} \alpha$ and $M+u_0$ in $V^{\pi_{v_0}^{-k}}$, one gets $x\in (M+u_0) \subset (L+u_0)$ with $x\in^* M_v$ for $v\not\in T$ such that $q(x)=\alpha$. Since $ord_{v_0}(\alpha)=k$, one concludes that $x\in^* L_{v_0}$ as well. The proof is complete. \end{proof}

The weak version of Theorem \ref{defrep} was also proved in Theorem 4.9 of \cite{CO}. 

\begin{rem} It should be pointed out that Cassels' trick (Lemma 9.1 of Chapter 11 in \cite{Cas}) can be applied only for even dimensional quadratic spaces.
For Statement (LT) with $dim(V)\geq 5$, one needs a slight different proof (see p.236-p.241 in \cite{Cas} for weak version). In \cite{JK}, J\"ochner and Kitaoka generalized such result to a lattice of rank $n$ represented by a lattice of rank $m\geq 2n+3$ with congruent condition under a bounded restriction. In \cite{J}, J\"ochner further removed such bounded restriction for $n=2$.
\end{rem}

\bigskip



\begin{bibdiv}

\begin{biblist}



\bib {Cas}{book}{
    author={J.W.S. Cassels},
     title={Rational quadratic forms},
     publisher={Academic Press},
     place={London},
      journal={ },
            series={London Mathematical Society Monographs},
    volume={13},
    date={1978},
    number={ },
     pages={},
}

\bib{CO} {article} {
    author={W.K.Chan},
    author={B.K.Oh},
    title={Representations of integral quadratic polynomials},
    journal={Contemp. Math.},
    volume={587},
    date={2013},
    Pages={31-46},
}



\bib{CTH} {article} {
    author={J.-L. Colliot-Th\'el\`ene},
    author={David Harari},
    title={Approximation forte en famille},
    journal={preprint},
    volume={},
    date={},
    Pages={},
}




\bib{CTX} {article} {
    author={J.-L. Colliot-Th\'el\`ene},
    author={F. Xu},
    title={Brauer-Manin obstruction for integral points of homogeneous spaces and
         representations by integral quadratic forms},
    journal={Compositio Math.},
    volume={145},
    date={2009},
    Pages={309-363},
}

\bib{CTX1} {article} {
    author={J.-L. Colliot-Th\'el\`ene},
    author={F. Xu},
    title={Strong approximation for total space of certain quadric fibrations},
    journal={Acta Arithmetica},
    volume={157},
    date={2013},
    Pages={169-199},
}





\bib {Ei}{book}{
    author={M.Eichler},
     title={Quadratische Formen und orthogonale Gruppen},
     publisher={Springer },
     place={Berlin},
      journal={ },
            series={Grundlehren der Mathematik},
    volume={63},
    date={1952},
    number={ },
     pages={},
}




\bib{HJ}{article}{
author={J.S.Hsia}
author={M.J\"ochner}
title={Almost strong approximations for definite quadratic spaces}
journal={Invent.math.}
vol={129}
date={1997}
Pages={471-487}
}


\bib{HKK}{article} {
author={J.S.Hsia}
author={Y.Kitaoka}
author={M.Kneser}
title={
Representations of positive definite quadratic
forms},
journal={J. reine und angew.
Math.},
vol={301}
date={1978}
Pages={132-141}
}


\bib{HSX}{article} {
author={J.S.Hsia}
author={Y.Y.Shao}
author={F.Xu}
title={
Representations of indefinite quadratic
forms},
journal={J. reine und angew.
Math.},
vol={494}
date={1998}
Pages={129-140}
}

\bib{J}{article}{
author={M.J\"ochner}
title={Representations of positive definite quadratic forms with congruence and primitive conditions, II}
journal={J. Number Theory}
vol={50}
date={1995}
pages={145-153}
}



\bib{JK}{article}{
author={M.J\"ochner}
author={Y.Kitaoka}
title={Representations of positive definite quadratic forms with congruence and primitive conditions}
journal={J. Number Theory}
vol={48}
date={1994}
pages={88-101}
}


\bib{Ki}{article} {
author={Y.Kitaoka},
title={Representaions of quadratic forms},
journal={Nagoya Math. J.},
volume={69}
date={1978}
Pages={ 117-120},
}

\bib{Ki1}{book}{
author={Y.Kitaoka},
title={Arithmetic of Quadratic Forms},
publisher={Cambridge University Press},
journal={},
series={Cambridge Tracts in Mathematics},
vol={106},
date={1993},
number={},
pages={},
}

\bib{Kn}{article} {
author={M.Kneser},
title={Darstellungsmasse indefiniter quadratischer Formen},
journal={Math. Z.},
volume={77}
date={1961}
Pages={188-194},
}



\bib{N}{book}{
    author={ J.Neukirch},
    title={Algebraic Number Fields},
    volume={322},
    publisher={Springer},
    series={Grundlehren},
    date={1999},
}

\bib {OM}{book}{
    author={O.T. O'Meara},
     title={Introduction to quadratic forms},
     publisher={Springer},
     place={Berlin},
      journal={ },
            series={Grundlehren der Mathematik},
    volume={270},
    date={1971},
    number={ },
     pages={},
}

\bib{OM58} {article} {
    author={O.T.O'Meara},
    title={The integral representations of quadratic forms over local fields},
    journal={Amer. J. Math.},
    volume={80},
    date={1958},
    Pages={843-878},
}

\bib {PR}{book}{
    author={V.P. Platonov},
    author={A.S. Rapinchuk}
     title={Algebraic groups and number theory},
     publisher={Academic Press},
     place={},
      journal={ },
            series={},
    volume={},
    date={1994},
    number={ },
     pages={},
}


\bib{Rie} {article} {
    author={C.Riehm},
    title={On the integral representations of quadratic forms over local fields},
    journal={Amer. J. Math.},
    volume={86},
    date={1964},
    Pages={25-62},
}


\bib{SX} {article} {
    author={R.Schulze-Pillot},
    author={F.Xu}
    title={Representations by spinor genera of ternary quadratic forms},
    journal={Contemp.Math.},
    volume={344},
    date={2004},
    Pages={323-337},
}

\bib{Shi1} {article} {
    author={G.Shimura},
    title={Inhomogeneous quadratic forms and triangular numbers},
    journal={Amer. J. Math.},
    volume={126},
    date={2004},
    Pages={191-214},
}

\bib{Shi2} {article} {
    author={G.Shimura},
    title={Quadratic Diophantine equations and orders in quaternion algebras},
    journal={Amer. J. Math.},
    volume={128},
    date={2006},
    Pages={481-518},
}

\bib{Shi3} {article} {
    author={G.Shimura},
    title={Integer-valued quadratic forms and quadratic Diophantine equations.},
    journal={Doc. Math.},
    volume={11},
    date={2006},
    Pages={333-367},
}



\bib {Sko}{book}{
    author={A. N. Skorobogatov},
     title={Torsors and rational points},
     publisher={Cambridge University Press},
     place={},
      journal={ },
            series={Cambridge Tracts in Mathematics},
    volume={144},
    date={2001},
    number={ },
     pages={},
}


\bib {Sun07} {article}{
author={Z. W. Sun}
title={ Mixed sums of squares and triangular numbers}
journal={Acta Arith.}
volume={127}
date={2007}
pages={103-113}
}

\bib{Wat60} {article} {
    author={G.L.Watson},
    title={Quadratic diophantine equations},
    journal={Philos.Trans.Roy.Soc.London A},
    volume={253},
    date={1960},
    Pages={227-254},
}




\bib{Wat} {article} {
    author={G.L. Watson},
    title={Diophantine equations reducible to quadratics},
    journal={Proc. London Math. Soc.},
    volume={17},
    date={1967},
    Pages={26-44},
}

\bib{WX} {article} {
    author={D. Wei},
    author={F. Xu},
 title={Integral points for multi-norm tori},
  journal={Proc. London Math. Soc.},
    volume={104},
    number={3}
      date={2012},
    pages={1019-1044},
 }

\bib{Xu} {article} {
    author={F. Xu},
    title={On representations of spinor genera II},
    journal={Math. Ann.},
    volume={332},
    date={2005},
    Pages={37-53},
}



\end{biblist}
\end{bibdiv}
\end{document}